\documentclass[a4paper,12pt]{amsart}
\usepackage{amsmath,amssymb,amsthm,verbatim}
\usepackage[mathscr]{eucal}
\usepackage{mathpazo}
\usepackage[T1]{fontenc}

\theoremstyle{plain}
\newtheorem*{Theorem}{Main Theorem}
\newtheorem{theorem}{Theorem}

\newtheorem{proposition}[theorem]{Proposition}

\theoremstyle{definition}
\newtheorem*{remark}{Remark}

\newcommand{\M}{\mathscr{M}}

\newcommand{\BH}{\mathbb{B}(\mathcal{H})}
\newcommand{\Ha}{\mathcal{H}}
\newcommand{\F}{\Phi}
\newcommand{\al}{\alpha}
\newcommand{\1}{\mathbb{1}}
\newcommand{\ro}{\rho}

\newcommand{\tr}{\operatorname{tr}}
\renewcommand{\geq}{\geqslant}
\renewcommand{\leq}{\leqslant}
\renewcommand{\epsilon}{\varepsilon}
\newcommand{\s}{\operatorname{supp}}

\begin{document}
\title[Mappings preserving Renyi's entropy]{Mappings preserving quantum Renyi's entropies in von Neumann algebras}
\author{Andrzej \L uczak}\author{Hanna Pods\k{e}dkowska}\author{Rafa{\l} Wieczorek}
\address{Faculty of Mathematics and Computer Science\\
         \L\'od\'z University\\
         ul. S. Banacha 22\\
         90-238 \L\'od\'z, Poland}

\email[Andrzej \L uczak]{andrzej.luczak@wmii.uni.lodz.pl}
\email[Hanna Pods\k{e}dkowska]{hanna.podsedkowska@wmii.uni.lodz.pl}
\email[Rafa{\l} Wieczorek]{rafal.wieczorek@wmii.uni.lodz.pl}
\thanks{}
\date{}
\begin{abstract}
 We investigate the situation when a normal positive linear unital map on a semifinite von Neumann algebra leaving the trace invariant does not change fixed quantum Renyi's entropy of the density of a normal state. It is also shown that such a map does not change the entropy of any density if and only if it is a Jordan *-isomorphism on the algebra.
\end{abstract}
\maketitle

\section*{Introduction}
In the paper, the question of invariance of quantum Renyi's entropies of a density under the action of a normal positive linear unital map $\F$ is addressed in the case of a semifinite von Neumann algebra. This can be seen as a follow-up to the investigation in \cite{LP} where such a problem was analysed for the Segal entropy. It is interesting that in both the cases: Segal's entropy and any single Renyi's entropy the answer is the same, namely, $\F$ restricted to the algebra generated by the density is a *-isomorphism. The question of invariance of Renyi's entropy for every density is also considered, and it turns out that this holds if and only if $\F$ is a Jordan *-isomorphism on the algebra.

In our analysis, we aim at full generality considering densities which may be unbounded. For such densities a number of specific Jensen's inequalities are derived. These inequalities are well-known, in a more general form, for bounded operators.

It is interesting to observe that passing from the case of bounded densities to unbounded ones requires some arguments referring to the algebra of measurable operators, in particular, convergence in measure for such operators is widely exploited.

\section{Preliminaries and notation}
Let $\M$ be a semifinite von Neumann algebra of operators acting on a Hilbert space $\Ha$ with a normal semifinite faithful trace $\tau$, identity $\1$, and predual $\M_*$. By $\M^+$
we shall denote the set of positive operators in $\M$, and by $\M_*^+$ --- the set of positive functionals in $\M_*$. These functionals will be sometimes referred to as
(non-normalised) states.

The algebra of \emph{measurable operators} $\widetilde{\M}$ is defined as a topological ${}^*$-algebra of densely defined closed operators on $\mathcal{H}$ affiliated with $\M$ with
strong addition $\dotplus$ and strong multiplication $\cdot$, i.e.
\[
 x\dotplus y=\overline{x+y},\qquad x\cdot y=\overline{xy},\qquad x,y\in\widetilde{\M},
\]
where $\overline{x+y}$ and $\overline{xy}$ are the closures of the corresponding operators defined by addition and composition respectively on the natural domains given by the
intersections of the domains of the $x$ and $y$ and of the range of $y$ and the domain of $x$ respectively. In what follows, we shall omit the dot in the symbols of these operations
and write simply $x+y$ and $xy$ to denote $x\dotplus y$ and $x\cdot y$. The translation-invariant measure topology is defined by a fundamental system of neighbourhoods of $0$, $\{N(\epsilon,\delta): \epsilon,\delta>0\}$, given by
\begin{align*}
 N(\epsilon,\delta)=\{x\in\widetilde{\M}:&\text{ there exists a projection $p$ in $\M$ such that}\\ &xp\in\M,\quad\|xp\|\leq\epsilon\quad\text{and} \quad\tau(\1-p)\leq\delta\}.
\end{align*}
The following `technical' form of convergence in measure proved in \cite[Proposition 2.7]{Y} is useful. Let
\[
 |x_n-x|=\int_0^\infty t\,e_n(dt)
\]
be the spectral representation of $|x_n-x|$ with spectral measures $e_n$ taking values in $\M$ since $x_n-x$, and thus $|x_n-x|$, are affiliated with $\M$. Then $x_n\to x$ \emph{in measure} if and only if for each $\epsilon>0$
\begin{equation}\label{cim}
 \tau(e_n([\epsilon,+\infty)))\to 0.
\end{equation}

The domain of a linear operator $x$ on $\Ha$ will be denoted by $\mathcal{D}(x)$.

For each $\ro\in\M_*$, there is a measurable operator $h$ such that
\[
 \ro(x)=\tau(xh)=\tau(hx), \quad x\in\M.
\]
The space of all such operators is denoted by $L^1(\M,\tau)$, and the correspondence above is one-to-one and isometric, where the norm on $L^1(\M,\tau)$, denoted by $\|\cdot\|_1$, is
defined as
\[
 \|h\|_1=\tau(|h|), \quad h\in L^1(\M,\tau).
\]
(In the theory of noncommutative $L^p$-spaces for semifinite von Neumann algebras, it it shown that $\tau$ can be extended to the $h$'s as above; see  e.g. \cite{N,T,Y} for a detailed
account of this theory.) Moreover, to hermitian functionals in $\M_*$ correspond selfadjoint operators in $L^1(\M,\tau)$, and to states in $\M_*$ --- positive operators in
$L^1(\M,\tau)$.

Chebyschev's inequality
\[
 \tau(e([\epsilon,+\infty)))\leq\frac{\tau(|x|)}{\epsilon}=\frac{\|x\|_1}{\epsilon},
\]
where
\[
 |x|=\int_0^\infty t\,e(dt)
\]
is the spectral representation, together with the relation \eqref{cim} show that convergence in $\|\cdot\|_1$-norm implies convergence in measure.

For $\ro\in\M_*^+$, the corresponding element in $L^1(\M,\tau)$ is called the \emph{density} of $\ro$ and is denoted by $h_\ro$.

Let $\F\colon\M\to\M$ be a normal positive linear unital mapping such that $\tau\circ\F=\tau$ (such maps are called \emph{channels}, usually with an additional assumption of complete positivity imposed). It is easy to see that then $\F$ can be extended to a bounded map on $L^1(\M,\tau)$ (cf. \cite[Lemma 8]{LP}).

Let $\al\in(0,1)\cup(1,+\infty)$. The $\al$-\emph{Renyi entropy of} $\ro$, denoted by $S_\al(\ro)$, is defined as
\[
 S_\al(\ro)=\frac{1}{1-\al}\log\frac{\tau(h_\ro^\al)}{\tau(h_\ro)},
\]
under the assumption $h_\ro^\al\in L^1(\M,\tau)$, i.e. for the spectral representation of $h_\ro$
\begin{equation}\label{spec}
 h_\ro=\int_0^\infty t\,e(dt),
\end{equation}
we have
\[
 S_\al(\ro)=\frac{1}{1-\al}\log\frac{\int_0^\infty t^\al\,\tau(e(dt))}{\int_0^\infty t\,\tau(e(dt))}.
\]
Accordingly, we define $\al$-Renyi's entropy for $0\leq h\in L^1(\M,\tau)$ such that $h^\al\in L^1(\M,\tau)$, by the formula
\[
 S_\al(h)=\frac{1}{1-\al}\log\frac{\tau(h^\al)}{\tau(h)}=\frac{1}{1-\al}\log\frac{\int_0^\infty t^\al\,\tau(e(dt))}{\int_0^\infty t\,\tau(e(dt))},
\]
where $h$ has the spectral representation as in \eqref{spec}. For a map $\F$ as before, we are interested in the problem when this map does not change Renyi's entropy of a density $h$, i.e. when the equality
\[
 S_\al(h)=S_\al(\F(h))
\]
holds.

\section{Main theorem}
In this section, we aim at proving the following main result of the paper.
\begin{Theorem}\label{MT}
Let $\F$ be a normal linear unital mapping on $\M$ such that $\tau\circ\F=\tau$, and let for an arbitrary fixed $\al\in(0,1)\cup(1,+\infty)$ and a density $h$ such that $h^\al\in L^1(\M,\tau)$
\[
 S_\al(h)=S_\al(\F(h)).
\]
Then $\F|W^*(h)$ is a *-isomorphism for $\al\leq2$, and the same holds for $\al>2$ if we additionally assume that there exists $1<\gamma\leq2$ such that\\ $h^\gamma\in L^1(\M,\tau)$
(this assumption is automatically satisfied for a finite algebra).
\end{Theorem}
Let us first explain the general idea of the proof. The invariance of entropy is obviously equivalent to the equality
\begin{equation}\label{tau=}
 \tau(\F(h^\al))=\tau(h^\al)=\tau(\F(h)^\al).
\end{equation}
From this equality, the equality
\begin{equation}\label{F=}
 \F(h^\al)=\F(h)^\al
\end{equation}
is attempted to be derived, leading, in turn, to the final conclusion. It should be noted that this equality for $\al\in(0,1)\cup(1,2]$ implies for \emph{bounded} $h$ our final conclusion, which follows from \cite{P}. Observe that in some cases the derivation of the relation \eqref{F=} from \eqref{tau=} is simple. Namely, for $0<\al<1$, the function $[0,+\infty)\ni t\mapsto t^\al$ is operator concave, and Jensen's inequality for \emph{bounded} operators yields the inequality
\[
 \F(h^\al)\leq\F(h)^\al.
\]
For $1<\al\leq2$ the function $[0,+\infty)\ni t\mapsto t^\al$ is operator convex, and again Jensen's inequality for \emph{bounded} operators yields the inequality
\[
 \F(h^\al)\geq\F(h)^\al.
\]
In both the cases above, the faithfulness of $\tau$ together with the relation \eqref{tau=} yield the equality \eqref{F=}. In the general case of not necessarily bounded densities $h$, we could attempt to derive appropriate Jensen's inequalities (which we actually do for $1<\al\leq2$) however, a simpler way is possible. Namely, it turns out that in the case $0<\al<1$, \emph{the equality} \eqref{tau=} \emph{alone is sufficient for proving the thesis}. For $\al>1$, we follow the route described above and obtain the equality \eqref{F=}. From this we get
\[
 \F(h^\al)^\frac{1}{\al}=\F(h)=\F\big((h^\al)^\frac{1}{\al}\big)
\]
which is the equality \eqref{F=} with $h^\al$ instead of $h$ and $\frac{1}{\al}<1$ instead of $\al$, and thus we get the equality \eqref{tau=} in the form
\[
 \tau\big(\F\big((h^\al)^\frac{1}{\al}\big)\big)=\tau\big(\F(h^\al)^\frac{1}{\al}\big).
\]
As pointed out above, this gives the claim upon observing that\\ $W^*(h^\al)=W^*(h)$. Consequently, the proof of the theorem is divided into three parts: 1. Case $0<\al<1$, in which the thesis is proved;\\ 2. Case $1<\al<2$ and 3. Case $\al\geq2$, in which cases the equality $\F(h^\gamma)=\F(h)^\gamma$ for some $\gamma>1$ is obtained.
\begin{proof}[Proof for $0<\al<1$]
Let $t\geq0$. We have the formula
\begin{equation}\label{t}
 t^\al=\frac{\sin\al\pi}{\pi}\int_0^\infty\frac{1}{s^{1-\al}}\frac{t}{s+t}\,ds.
\end{equation}
Let $0\leq h\in L^1(\M,\tau)$. Since for $s>0$, we have
\[
 \|h(s\1+h)^{-1}\|_1\leq\|h\|_1\|(s\1+h)^{-1}\|_\infty\leq\frac{\|h\|_1}{s},
\]
there exists, for each $m>0$, Bochner's integral
\[
 z_m(h)=\frac{\sin\al\pi}{\pi}\int_m^\infty\frac{1}{s^{1-\al}}h(s\1+h)^{-1}\,ds
\]
with values in $L^1(\M,\tau)$. Let
\begin{equation}\label{sp}
 h=\int_0^\infty t\,e(dt)
\end{equation}
be the spectral representation of $h$. Then
\[
 \tau(h(s\1+h)^{-1})=\int_0^\infty\frac{t}{s+t}\,\tau(e(dt)),
\]
and on account of the Lebesgue monotone convergence theorem, Fubini's theorem, and \eqref{t} we get for $m\searrow0$
\begin{align*}
 &\tau(z_m(h))=\tau(\F(z_m(h)))=\frac{\sin\al\pi}{\pi}\int_m^\infty\frac{1}{s^{1-\al}}\tau(h(s\1+h)^{-1})\,ds\\
 =&\frac{\sin\al\pi}{\pi}\int_m^\infty\frac{1}{s^{1-\al}}\int_0^\infty\frac{t}{s+t}\,\tau(e(dt)\,ds\\
 &\nearrow\frac{\sin\al\pi}{\pi}\int_0^\infty\frac{1}{s^{1-\al}}\int_0^\infty\frac{t}{s+t}\,\tau(e(dt)\,ds\\
 =&\frac{\sin\al\pi}{\pi}\int_0^\infty\int_0^\infty\frac{1}{s^{1-\al}}\frac{t}{s+t}\,ds\,\tau(e(dt)=\int_0^\infty t^\al\tau(e(dt))=\tau(h^\al).
\end{align*}
By the same token, we obtain
\[
 \tau(z_m(\F(h)))\nearrow\tau(\F(h)^\al).
\]
Since
\[
 \frac{t}{s+t}=1-\frac{s}{s+t},
\]
we infer by virtue of \cite[Theorem 7]{LP} that
\begin{equation}\label{cc}
 \F(h(s\1+h)^{-1})\leq\F(h)(s\1+\F(h))^{-1},
\end{equation}
thus
\begin{align*}
 &\tau(z_m(\F(h)))-\tau(z_m(h))=\tau(z_m(\F(h)))-\tau(\F(z_m(h)))\\
 =&\frac{\sin\al\pi}{\pi}\int_m^\infty\frac{1}{s^{1-\al}}\tau(\F(h)(s\1+\F(h))^{-1}-\F(h(s\1+h)^{-1}))\,ds,
\end{align*}
and the function under the integral sign is nonnegative which means that $\tau(z_m(\F(h)))-\tau(z_m(h))$ increases as $m\searrow0$. But
\[
 \tau(z_m(\F(h)))-\tau(z_m(h))\nearrow\tau(\F(h)^\al)-\tau(h^\al)=0,
\]
consequently,
\[
 \int_m^\infty\frac{1}{s^{1-\al}}\tau(\F(h)(s\1+\F(h))^{-1}-\F(h(s\1+h)^{-1}))\,ds=0
\]
and the non-negativity of the function under the integral sign yields
\[
 \tau(\F(h)(s\1+\F(h))^{-1}-\F(h(s\1+h)^{-1}))=0 \quad \textit{almost everywhere},
\]
which by virtue of the inequality \eqref{cc} gives
\begin{equation}\label{F}
 \F(h)(s\1+\F(h))^{-1}=\F(h(s\1+h)^{-1}) \quad \textit{almost everywhere},
\end{equation}
i.e. \textit{everywhere} by continuity. Since
\[
 x(s\1+x)^{-1}=\1-s(s\1+x)^{-1}
\]
for $x\geq0$, the equality \eqref{F} yields
\[
 (s\1+\F(h))^{-1}=\F((s\1+h)^{-1}),
\]
and now repeating word for word the reasoning as in \cite[Theorem~12]{LP} we obtain the claim.
\end{proof}

\begin{proof}[Proof for $1<\al\leq2$]
Assume first that $\al<2$, and let $\beta=\al-1$. We have
\begin{equation}\label{t1}
 t^{1+\beta}=\frac{\sin\beta\pi}{\pi}\int_0^\infty\frac{1}{s^{1-\beta}}\frac{t^2}{s+t}\,ds,
\end{equation}
and the estimate
\[
 \|h^2(s\1+h)^{-1}\|_1\leq\|h\|_1\|h(s\1+h)^{-1}\|_\infty\leq\|h\|_1.
\]
For $0<m<M$ set
\[
 z_m^M(h)=\frac{\sin\beta\pi}{\pi}\int_m^M\frac{1}{s^{1-\beta}}h^2(s\1+h)^{-1}\,ds,
\]
which is Bochner's integral in $L^1(\M,\tau)$. Assume for a moment that $h$ is bounded, $h\in\M$. Then since
\[
 \|h^2(s\1+h)^{-1}\|_\infty=\frac{\|h\|_\infty^2}{s+\|h\|_\infty},
\]
we infer that $z_m^M(h)$ is also Bochner's integral in $\M$, and for each $\xi\in\Ha$ we have on account of Fubini's theorem
\begin{align*}
 \langle z_m^M(h)\xi|\xi\rangle&=\frac{\sin\beta\pi}{\pi}\int_m^M\frac{1}{s^{1-\beta}}\langle h^2(s\1+h)^{-1}\xi|\xi\rangle\,ds\\
 &=\frac{\sin\beta\pi}{\pi}\int_m^M\frac{1}{s^{1-\beta}}\int_0^\infty\frac{t^2}{s+t}\,\|e(dt)\xi\|^2\,ds\\
 &\leq\frac{\sin\beta\pi}{\pi}\int_0^\infty\frac{1}{s^{1-\beta}}\int_0^\infty\frac{t^2}{s+t}\,\|e(dt)\xi\|^2\,ds\\
 &=\frac{\sin\beta\pi}{\pi}\int_0^\infty\int_0^\infty\frac{1}{s^{1-\beta}}\frac{t^2}{s+t}\,ds\,\|e(dt)\xi\|^2\\
 &=\int_0^\infty t^{1+\beta}\,\|e(dt)\xi\|^2=\langle h^{1+\beta}\xi|\xi\rangle,
\end{align*}
showing that
\begin{equation}\label{ogr}
 z_m^M(h)\leq h^{1+\beta}.
\end{equation}
Now we drop the assumption that $h$ is bounded and put
\[
 h_n=\int_0^nt\,e(dt).
\]
We have
\[
 h_n^2(s\1+h_n)^{-1}\leq h_{n+1}^2(s\1+h_{n+1})^{-1}\leq h^2(s\1+h)^{-1},
\]
consequently,
\[
 z_m^M(h_n)\leq z_m^M(h_{n+1})\leq z_m^M(h).
\]
Moreover, by virtue of Fubini's theorem and the Lebesgue monotone convergence theorem, we obtain as $n\to\infty$
\begin{align*}
 \tau(z_m^M(h_n))&=\frac{\sin\beta\pi}{\pi}\int_m^M\frac{1}{s^{1-\beta}}\tau(h_n^2(s\1+h_n)^{-1})\,ds\\
 =&\frac{\sin\beta\pi}{\pi}\int_m^M\frac{1}{s^{1-\beta}}\int_0^n\frac{t^2}{s+t}\,\tau(e(dt)\,ds\\
 &=\frac{\sin\beta\pi}{\pi}\int_0^n\int_m^M\frac{1}{s^{1-\beta}}\frac{t^2}{s+t}\,ds\,\tau(e(dt)\\
 &\nearrow\frac{\sin\beta\pi}{\pi}\int_0^\infty\int_m^M\frac{1}{s^{1-\beta}}\frac{t^2}{s+t}\,ds\,\tau(e(dt)\\
 &=\frac{\sin\beta\pi}{\pi}\int_m^M\frac{1}{s^{1-\beta}}\int_0^\infty\frac{t^2}{s+t}\,\tau(e(dt)\,ds\\
 &=\frac{\sin\beta\pi}{\pi}\int_m^M\frac{1}{s^{1-\beta}}\tau(h^2(s\1+h)^{-1})\,ds=\tau(z_m^M(h)).
\end{align*}
This shows that for $n\to\infty$ we have
\begin{equation}\label{z}
 \begin{aligned}
  \|z_m^M(h)-z_m^M(h_n)\|_1&=\tau(z_m^M(h)-z_m^M(h_n))\\
  &=\tau(z_m^M(h))-\tau(z_m^M(h_n))\to0.
 \end{aligned}
\end{equation}
Further we have on account of \eqref{ogr}
\[
 z_m^M(h_n)\leq h_n^{1+\beta}\leq h^{1+\beta},
\]
which together with the relation \eqref{z} shows that
\[
 z_m^M(h)\leq h^{1+\beta}.
\]
Now we get, using the Lebesgue monotone convergence theorem with $m\searrow0$, $M\nearrow\infty$,
\begin{align*}
 \tau(z_m^M(h))&=\frac{\sin\beta\pi}{\pi}\int_m^M\frac{1}{s^{1-\beta}}\tau(h^2(s\1+h)^{-1})\,ds\\
 =&\frac{\sin\beta\pi}{\pi}\int_m^M\frac{1}{s^{1-\beta}}\int_0^\infty\frac{t^2}{s+t}\,\tau(e(dt))\,ds\\
 &=\frac{\sin\beta\pi}{\pi}\int_0^\infty\int_m^M\frac{1}{s^{1-\beta}}\frac{t^2}{s+t}\,ds\,\tau(e(dt)\\
 &\nearrow\frac{\sin\beta\pi}{\pi}\int_0^\infty\int_0^\infty\frac{1}{s^{1-\beta}}\frac{t^2}{s+t}\,ds\,\tau(e(dt)\\
 &=\int_0^\infty t^{1+\beta}\,\tau(e(dt))=\tau(h^{1+\beta}),
\end{align*}
showing that with $m\searrow0$, $M\nearrow\infty$
\[
 \|h^{1+\beta}-z_m^M(h)\|_1=\tau(h^{1+\beta}-z_m^M(h))=\tau(h^{1+\beta})-\tau(z_m^M(h))\to0,
\]
i.e.
\[
 z_m^M(h)\to h^{1+\beta} \quad \textit{in $\|\cdot\|_1$-norm}.
\]
Taking in the above reasoning $\F(h)$ in place of $h$, we obtain the relation
\[
 z_m^M(\F(h))\to\F(h)^{1+\beta} \quad \textit{in $\|\cdot\|_1$-norm}.
\]

Now we are in a position to prove the following version of Jensen's inequality
\[
 \F(h^{1+\beta})\geq\F(h)^{1+\beta}.
\]
Since for $x\geq0$
\[
 x^2(s\1+x)^{-1}=x-s\1+s^2(s\1+x)^{-1},
\]
we get by virtue of \cite[Theorem 7]{LP} that
\begin{equation}\label{F1}
 \begin{aligned}
  &\F(h^2(s\1+h)^{-1})-\F(h)^2(s\1+\F(h))^{-1}\\
  =&s^2(\F((s\1+h)^{-1})-(s\1+\F(h))^{-1})\geq0.
 \end{aligned}
\end{equation}
From the inequality \eqref{F1}, we obtain
\begin{align*}
 \F(z_m^M(h))&=\frac{\sin\beta\pi}{\pi}\F\Big(\int_m^M\frac{1}{s^{1-\beta}}h^2(s\1+h)^{-1}\,ds\Big)\\
 &=\frac{\sin\beta\pi}{\pi}\int_m^M\frac{1}{s^{1-\beta}}\F(h^2(s\1+h)^{-1})\,ds\\
 &\geq\frac{\sin\beta\pi}{\pi}\int_m^M\frac{1}{s^{1-\beta}}\F(h)^2(s\1+\F(h))^{-1})\,ds=z_m^M(\F(h)),
\end{align*}
and taking into account the relations
\[
 z_m^M(h)\to h^{1+\beta} \quad \textit{in $\|\cdot\|_1$-norm}
\]
and
\[
 z_m^M(\F(h))\to\F(h)^{1+\beta} \quad \textit{in $\|\cdot\|_1$-norm},
\]
together with the continuity of $\F$ in the $\|\cdot\|_1$-norm we prove the claim.

Now let $\al=2$. We want to prove that
\[
 \F(h^2)\geq\F(h)^2.
\]
With the notation as before we have
\[
 h_n=\int_0^nt\,e(dt)\nearrow\int_0^\infty t\,e(dt)=h \quad \textit{in $\|\cdot\|_1$-norm},
\]
and
\[
 h_n^2=\int_0^nt^2\,e(dt)\nearrow\int_0^\infty t^2\,e(dt)=h^2 \quad \textit{in $\|\cdot\|_1$-norm},
\]
thus
\[
 \F(h_n)\to\F(h) \quad \textit{in $\|\cdot\|_1$-norm}
\]
and
\[
 \F(h_n^2)\to\F(h^2) \quad \textit{in $\|\cdot\|_1$-norm}.
\]
In particular,
\[
 \F(h_n)\to\F(h) \quad \textit{in measure}
\]
and
\[
 \F(h_n^2)\to\F(h^2) \quad \textit{in measure}.
\]
Since multiplication is jointly continuous in the measure topology, we get
\[
 \F(h_n)^2\to\F(h)^2 \quad \textit{in measure}.
\]
Jensen's inequality for bounded operators yields
\[
 \F(h_n^2)\geq\F(h_n)^2,
\]
and passing to the limit in measure on both sides of the above inequality proves the claim.

Putting together the considerations above, we obtain the following Jensen inequality for $1<\al\leq2$ and $h\in L^1(\M,\tau)$, $h^\al\in L^1(\M,\tau)$,
\begin{equation}\label{JI}
 \F(h^\al)\geq\F(h)^\al.
\end{equation}
Consequently, we have for $1<\al\leq2$
\[
 \tau\big(\F(h^\al)-\F(h)^\al\big)=\tau\big(\F(h^\al)\big)-\tau\big(\F(h)^\al\big)=0,
\]
and the faithfulness of $\tau$ gives
\[
 \F(h^\al)=\F(h)^\al. \qedhere
\]
\end{proof}
Before passing to the case $\al>2$ we need some auxiliary results.
\begin{proposition}\label{zb}
Let $0<\al<1$, and let $0\leq x_n\in L^1(\M,\tau)$ be such that $x_n\nearrow x\in L^1(\M,\tau)$ in $\|\cdot\|_1$-norm. Then $x_n^\al\nearrow x^\al$ in $\|\cdot\|_1$-norm.
\end{proposition}
\begin{proof}
For arbitrary $0\leq u\in L^1(\M,\tau)$, and $0<m<M$ set
\[
 \widetilde{z}_m^M(u)=\frac{\sin\al\pi}{\pi}\int_m^M\frac{1}{s^{1-\al}}u(s\1+u)^{-1}\,ds.
\]
The estimates
\[
 \|u(s\1+u)^{-1}\|_1\leq\frac{\|u\|_1}{s} \qquad \text{and} \qquad \|u(s\1+u)^{-1}\|_\infty\leq1
\]
show that the integral above is Bochner's integral in $L^1(\M,\tau)$ as well as in $\M$. For arbitrary $\xi\in\mathcal{D}(u^\al)$, we have, for the spectral representation
\[
 u=\int_0^\infty t\,p(dt)
\]
of $u$, the following relation
\begin{align*}
 &\langle\widetilde{z}_m^M(u)\xi|\xi\rangle=\frac{\sin\al\pi}{\pi}\int_m^M\frac{1}{s^{1-\al}}\langle u(s\1+u)^{-1}\xi|\xi\rangle\,ds\\
 &=\frac{\sin\al\pi}{\pi}\int_m^M\frac{1}{s^{1-\al}}\int_0^\infty\frac{t}{s+t}\,\|p(dt)\xi\|^2\,ds\\
 &\leq\frac{\sin\al\pi}{\pi}\int_0^\infty\frac{1}{s^{1-\al}}\int_0^\infty\frac{t}{s+t}\,\|(p(dt)\xi\|^2\,ds\\
 &=\frac{\sin\al\pi}{\pi}\int_0^\infty\int_0^\infty\frac{1}{s^{1-\al}}\frac{t}{s+t}\,ds\,\|(p(dt)\xi\|^2\\
 &=\int_0^\infty t^\al\|(p(dt)\xi\|^2=\langle u^\al\xi|\xi\rangle,
\end{align*}
showing that
\begin{equation}\label{z-u}
 \widetilde{z}_m^M(u)\leq u^\al.
\end{equation}
Moreover, we obtain as $m\searrow0$, $M\nearrow\infty$,
\begin{align*}
 &\tau(\widetilde{z}_m^M(u))=\frac{\sin\al\pi}{\pi}\int_m^M\frac{1}{s^{1-\al}}\tau(u(s\1+u)^{-1})\,ds\\
 =&\frac{\sin\al\pi}{\pi}\int_m^M\frac{1}{s^{1-\al}}\int_0^\infty\frac{t}{s+t}\,\tau(p(dt))\,ds\\
 &\nearrow\frac{\sin\al\pi}{\pi}\int_0^\infty\frac{1}{s^{1-\al}}\int_0^\infty\frac{t}{s+t}\,\tau(p(dt))\,ds\\
 =&\frac{\sin\al\pi}{\pi}\int_0^\infty\int_0^\infty\frac{1}{s^{1-\al}}\frac{t}{s+t}\,ds\,\tau(p(dt))=\int_0^\infty t^\al\tau(p(dt))=\tau(u^\al),
\end{align*}
which yields
\[
 \widetilde{z}_m^M(u)\to u^\al \quad \textit{in $\|\cdot\|_1$-norm}.
\]
For $0\leq u_1\leq u_2$ we have $s\1+u_1\leq s\1+u_2$ thus
\[
 (s\1+u_1)^{-1}\geq(s\1+u_2)^{-1},
\]
so
\[
 u_1(s\1+u_1)^{-1}=\1-s(s\1+u_1)^{-1}\leq\1-s(s\1+u_2)^{-1}=u_2(s\1+u_2)^{-1},
\]
which yields
\[
 \widetilde{z}_m^M(u_1)\leq\widetilde{z}_m^M(u_2).
\]
Passing to the limit with $m\searrow0$, $M\nearrow\infty$, we get
\begin{equation}\label{opmon}
 u_1^\al\leq u_2^\al,
\end{equation}
which shows the operator monotonicity of the function\\ $[0,+\infty)\ni t\mapsto t^\al$ for elements from $L^1(\M,\tau)$ (a result well-known for bounded operators).

Let $x_n$ and $x$ be as in the assumption. Then
\[
  \widetilde{z}_m^M(x_n)\leq\widetilde{z}_m^M(x).
\]
Further we have
\begin{align*}
 &x_n(s\1+x_n)^{-1}-x(s\1+x)^{-1}\\
 =&(s\1+x_n)^{-1}(x_n(s\1+x)-(s\1+x_n)x)(s\1+x)^{-1}\\
 =&s(s\1+x_n)^{-1}(x_n-x)(s\1+x_n)^{-1},
\end{align*}
thus
\begin{align*}
 &\|x_n(s\1+x_n)^{-1}-x(s\1+x)^{-1}\|_1\\
 \leq&s\|(s\1+x_n)^{-1}\|_\infty\|x_n-x\|_1\|(s\1+x_n)^{-1}\|_\infty\leq\frac{\|x_n-x\|_1}{s},
\end{align*}
which shows that
\[
 x_n(s\1+x_n)^{-1}\nearrow x(s\1+x)^{-1} \quad \textit{in $\|\cdot\|_1$-norm}.
\]
From the Lebesgue monotone convergence theorem we obtain as $n\to\infty$
\begin{align*}
 &\tau(\widetilde{z}_m^M(x_n))=\frac{\sin\al\pi}{\pi}\int_m^M\frac{1}{s^{1-\al}}\tau(x_n(s\1+x_n)^{-1})\,ds\\
 &\nearrow\frac{\sin\al\pi}{\pi}\int_m^M\frac{1}{s^{1-\al}}\tau(x(s\1+x)^{-1})\,ds=\tau(\widetilde{z}_m^M(x)),
\end{align*}
i.e.
\[
 \widetilde{z}_m^M(x_n)\nearrow\widetilde{z}_m^M(x) \quad \textit{in $\|\cdot\|_1$-norm}.
\]

From the relation \eqref{z-u} and operator monotonicity we obtain
\[
 \tau(\widetilde{z}_m^M(x_n))\leq\tau(x_n^\al)\leq\tau(x^\al),
\]
which gives
\begin{align*}
 \tau(\widetilde{z}_m^M(x))=\lim_{n\to\infty}\tau(\widetilde{z}_m^M(x_n))
 \leq\varliminf_{n\to\infty}\tau(x_n^\al)\leq\varlimsup_{n\to\infty}\tau(x_n^\al)\leq\tau(x^\al),
\end{align*}
and passing to the limit with $m\searrow0$, $M\nearrow\infty$, in the above inequality, we get
\begin{align*}
 \tau(x^\al)=\lim_{\substack{m\to0\\M\to\infty}}\tau(\widetilde{z}_m^M(x))
 \leq\varliminf_{n\to\infty}\tau(x_n^\al)\leq\varlimsup_{n\to\infty}\tau(x_n^\al)\leq\tau(x^\al),
\end{align*}
showing that
\[
 \lim_{n\to\infty}\tau(x_n^\al)=\tau(x^\al),
\]
which by virtue of the inequality $x_n^\al\leq x^\al$ proves the claim.
\end{proof}
Another version of Jensen's inequality is
\begin{proposition}\label{Jin2}
Let $\al>1$, and let $0\leq h\in L^1(\M,\tau)$ be such that\\ $h^\al\in L^1(\M,\tau)$. Then
\[
 \tau(\F(h^\al))\geq\tau(\F(h)^\al).
\]
\end{proposition}
\begin{proof}
We have
\[
 h_n=\int_0^nt\,e(dt)\nearrow\int_0^\infty t\,e(dt)=h \quad \textit{in $\|\cdot\|_1$-norm},
\]
and
\[
 h_n^\al=\int_0^nt^\al\,e(dt)\nearrow\int_0^\infty t^\al\,e(dt)=h^\al \quad \textit{in $\|\cdot\|_1$-norm},
\]
thus
\[
 \F(h_n)\nearrow\F(h) \quad \textit{in $\|\cdot\|_1$-norm}
\]
and
\[
 \F(h_n^\al)\nearrow\F(h^\al) \quad \textit{in $\|\cdot\|_1$-norm}.
\]
In particular,
\[
 \F(h_n)\to\F(h) \quad \textit{in measure}
\]
and
\[
 \tau(\F(h_n^\al))\to\tau(\F(h^\al)).
\]
Let $\al=r+\beta$, where $r=[\al]$, $0\leq\beta<1$. On account of Proposition \ref{zb} we have
\[
 \F(h_n)^\beta\to\F(h)^\beta \quad \textit{in $\|\cdot\|_1$-norm},
\]
so
\[
 \F(h_n)^\beta\to\F(h)^\beta \quad \textit{in measure},
\]
yielding, since multiplication is jointly continuous in the measure topology,
\[
 \F(h_n)^\al=\F(h_n)^r\F(h_n)^\beta\to\F(h)^r\F(h)^\beta=\F(h)^\al \quad \textit{in measure}.
\]
Since the function $[0,+\infty)\ni t\mapsto t^\al$ is convex, Jensen's inequality for bounded operators implies
\[
 \tau(\F(h_n^\al))\geq\tau(\F(h_n)^\al).
\]
For every $n$ we have
\[
 \|\F(h_n)^\al\|_1=\tau(\F(h_n)^\al)\leq\tau(\F(h_n^\al))=\tau(h_n^\al)\leq\tau(h^\al)<+\infty,
\]
and Fatou's lemma --- \cite[Theorem 2.9]{Y} --- yields
\begin{align*}
 \tau(\F(h)^\al)&=\|\F(h)^\al\|_1\leq\varliminf_{n\to\infty}\|\F(h_n)^\al\|_1\\
 &=\varliminf_{n\to\infty}\tau(\F(h_n)^\al)\leq\varliminf_{n\to\infty}\tau(\F(h_n^\al))=\tau(\F(h^\al)). \qedhere
\end{align*}
\end{proof}
\begin{proof}[Proof for $\al\geq2$]
Here we additionally assume that there exists\\ $1<\gamma\leq2$ such that $h^\gamma\in L^1(\M,\tau)$. Let
\[
 f(t)=t^{\frac{\al}{\gamma}}, \quad t\in[0,+\infty).
\]
By virtue of Proposition \ref{Jin2} \big(with $h$ replaced by $h^\gamma$ and $\al$ replaced by $\frac{\al}{\gamma}$\big), we have
\begin{align*}
 \tau\big(\F(h^\al)\big)&=\tau\big(\F\big((h^\gamma)^{\frac{\al}{\gamma}}\big)\big)
 \geq\tau\big(\big(\F(h^\gamma)^{\frac{\al}{\gamma}}\big)=\tau\big(f\big(\F\big(h^\gamma\big)\big)\big).
\end{align*}
Jensen's inequality \eqref{JI} is
\[
 \F\big(h^\gamma\big)\geq\F(h)^\gamma,
\]
and since $f$ is strictly increasing, continuous, convex, and $f(0)=0$, we obtain by virtue of \cite[Theorem 14]{HK}
\begin{equation}\label{>}
 \tau\big(f\big(\F\big(h^\gamma\big)\big)\big)\geq\tau\big(f\big(\F(h)^\gamma\big)\big)=\tau(\F(h)^\al).
\end{equation}
Since
\[
 \tau\big(\F(h^\al)\big)=\tau(\F(h)^\al),
\]
we have equality in the inequality \eqref{>} above which again by virtue of \cite[Theorem 14]{HK} yields the equality
\[
 \F\big(h^\gamma\big)=\F(h)^\gamma. \qedhere
\]
\end{proof}
\begin{remark}
Observe that if $\F$ is a *-isomorphism on the algebra $W^*(h)$ such that $\tau$ is $\F$-invariant, then for every bounded continuous function $f$ we have $\F(f(h))=f(\F(h))$, and it is easy to prove that for $\al\in(0,1)\cup(1,+\infty)$ such that $ h^\al\in L^1(\M,\tau)$ we have $\F(h^\al)=\F(h)^\al$, and thus $S_\al(h)=S_\al(\F(h))$.
\end{remark}

\section{Additional results}
In this section, we consider the situation when the map $\F$ does not change an arbitrary fixed entropy of every density.
\begin{theorem}
Let $\F$ be a normal linear unital mapping on $\M$ such that $\tau\circ\F=\tau$, and let for an arbitrary fixed $\al\in(0,1)\cup(1,+\infty)$
\[
 S_\al(h)=S_\al(\F(h))
\]
for every density $h$ such that $h^\al\in L^1(\M,\tau)$. Then $\F$ is a Jordan\\ *-isomorphism on $\M$.
\end{theorem}
\begin{proof}
Let $h\in\M$ be a density such that $h^\al\in L^1(\M,\tau)$ and, in the case $\al>2$, $h^\gamma\in L^1(\M,\tau)$ for some $1<\gamma\leq2$. From the Main Theorem it follows that $\F|W(h)$ is a *-isomorphism, in particular, we have
\[
 \F(h^2)=\F(h)^2.
\]
On account of \cite[Proposition 9]{LP}, for all $x\in\M$ the following equality holds
\begin{equation}\label{Jor}
 \F(x\circ h)=\F(x)\circ\F(h),
\end{equation}
where for $u,v\in\M$
\[
 u\circ v=\frac{uv+vu}{2},
\]
is the Jordan product.

Since such $h$'s as above are $\sigma$-weakly dense in $\M^+$, and $\F$ is normal, it follows that the equation \eqref{Jor} holds for all $h$ in $\M^+$, and consequently we have
\[
 \F(x\circ y)=\F(x)\circ\F(y).
\]
for all $x,y\in\M$, thus $\F$ is a Jordan *-homomorphism.

If $\F(x)=0$, then
\[
 0=\F(x)^*\F(x)+\F(x)\F(x)^*=\F(x^*x+xx^*)=\F(x^*x)+\F(xx^*)
\]
thus
\[
 \F(x^*x)=\F(xx^*)=0,
\]
and it follows that $\tau(x^*x)=\tau(\F(x^*x))=0$, i.e. $x=0$, showing that $\F$ is an isomorphism.
\end{proof}
\begin{remark}
If $\F$ is a Jordan *-isomorphism, then it is obviously\\ a *-isomorphism on the abelian algebra $W^*(h)$, thus the equality $S_\al(h)=S_\al(\F(h))$ follows as in the previous remark.
\end{remark}
At the end let us comment on a general situation. In many cases the quantities of the form
\begin{equation}\label{H}
 H(h)=\tau(f(h))
\end{equation}
where $h$ is a density, and $f$ is a continuous function, are considered (typical examples are $f(t)=t\log t$ leading to Segal's entropy or $f(t)=t^\al$ leading to Renyi's entropy). Then a natural question is the invariance of $H$ with respect to some map $\F$, i.e. the relation
\begin{equation}\label{inv}
 H(h)=H(\F(h)).
\end{equation}
Putting aside technical problems ($h$ may be unbounded), we have noticed that, in principle, if $\F$ is a normal Jordan *-isomorphism, then $\F(f(h))=f(\F(h))$, thus the additional assumption $\tau\circ\F=\tau$ yields the relation \eqref{inv}. Consequently, the condition: `$\F$ \emph{is a normal Jordan *-isomorphism such that} $\tau\circ\F=\tau$' is \emph{sufficient} for the invariance \eqref{inv}. A similar situation occurs when we deal with two densities. A typical example, which for simplicity we consider in finite dimension, is the relative entropy of two densities $h$ and $k$ defined as
\[
 D(h||k)=\begin{cases}
  \tr h(\log h-\log k), & \text{if $\s h\leq\s k$}\\
  +\infty, & \text{otherwise}
 \end{cases}.
\]
Now if $\F$ is a normal Jordan *-isomorphism leaving the trace `$\tr$' invariant, we have $\s\F(h)=\F(\s h)$ and $\s\F(k)=\F(\s k)$, thus $\s h\leq\s k$ if and only if $\s\F(h)\leq\s\F(k)$ in which case we get the equality
\begin{align*}
 &D(\F(h)||\F(k))=\tr\F(h)(\log\F(h)-\log\F(k))\\
 =&\tr(\F(h)\circ(\log\F(h)-\log\F(k)))\\
 =&\tr(\F(h)\circ(\F(\log h)-\F(\log k)))\\
 =&\tr(\F(h)\circ\F(\log h-\log k))=\tr\F(h\circ(\log h-\log k))\\
 =&\tr(h\circ(\log h-\log k))=\tr h(\log h-\log k)=D(h||k),
\end{align*}
showing that the same condition: `$\F$ \emph{is a normal Jordan *-isomorphism such that} $\tau\circ\F=\tau$' is again \emph{sufficient} for the invariance of the relative entropy. The above reasoning shows that some results which state e.g. that a map leaving the relative entropy, or the $H$ as in \eqref{H}, invariant must be of the form $\F(x)=uxu^*$ where $u$ is a unitary or antiunitary or an isometry are incorrect (a good example of such a Jordan *-isomorphism which does not change the entropy but is not of the form as above, is transposition in $\BH$ with respect to a given orthonormal basis).

\end{document}